\documentclass[12pt]{amsart}

\usepackage{amssymb}
\usepackage{enumerate}
\usepackage{amsmath}
\usepackage{epsfig}
\usepackage{amsthm}

\newtheorem{thm}{Theorem}[section]

\begin{document}

\title[Class number one criterion]{Class number one criterion for some non-normal totally real cubic fields}

   \author[Jun Ho Lee]
    {Jun Ho Lee}
    \address{School of Mathematics, Korea Institute for Advanced Study, Hoegiro 85, Dongdaemun-gu, Seoul 130-722, Republic of Korea}
     \email{jhleee@kias.re.kr}

    \thanks{{\it key words and phrases.} cubic fields, class number, zeta function.}

\maketitle

\begin{abstract}
 Let ${\{K_m\}_{m\geq 4}}$ be the family of non-normal totally real
        cubic number fields defined by the irreducible cubic polynomial
        $f_m(x)=x^3-mx^2-(m+1)x-1$, where $m$ is an integer with $m\geq
        4$. In this paper, we will give a class number one criterion for $K_m$.

\end{abstract}

\section{Introduction}
\label{sec:introduction}

    It has been known for a long time that there exists close connection between prime producing polynomials and class number one
    problem for some number fields. Rabinowitsch\cite{Ra} proved that for a prime number $q$,
    the class number of $\mathbb{Q}(\sqrt{1-4q})$ is equal to one if and only if $k^2+k+q$ is prime for every $k=0,1,\ldots,q-2.$
    For real quadratic fields, many authors\cite{BK1,BK2,Mo,Yo} considered the connection between prime producing polynomials
    and class number. For the simplest cubic fields, Kim and Hwang\cite{KH}
    gave a class number one criterion which is related to some prime producing polynomials. The aim of this paper is to give
    a class number one criterion for some non-normal totally real cubic fields. Its criterion provides some polynomials having almost prime
    values in a given interval. The method done in this paper is basically same as one in \cite{BK1,BK2,KH}.

    Let $\zeta_K(s)$ be the Dedekind zeta function of an algebraic number field $K$ and $\zeta_K(s,P)$ be the partial zeta
    function for the principal ideal class $P$ of $K$. Then we have
    \begin{displaymath}
    \zeta_K(-1) \leq \zeta_K(-1,P).
    \end{displaymath}
    Halbritter and Pohst\cite{HP} developed a method of expressing special values of the partial zeta
    functions of totally real cubic fields as a finite sum involving norm, trace, and 3-fold Dedekind
    sums. Their result has been exploited by Byeon\cite{B} to give an explicit formula for the values
    of the partial zeta functions of the simplest cubic fields. Kim and Hwang\cite{KH} gave a class number one
    criterion for the simplest cubic fields by estimating the value $\zeta_K(-1)$
    and combining Byeon's result. In this paper, we will do this kind of work in some non-normal totally
    real cubic fields. First, we apply Halbritter and Pohst's formula to our cubic fields, and then evaluate the upper bound
    of $\zeta_K(-1)$ by using Siegel's formula. Finally, combining this computation, we give a class number one criterion for some non-normal
    totally real cubic fields. Halbritter and Pohst\cite{HP} proved:

    \begin{thm}\label{theorem:1.1}
    Let $K$ be a totally real cubic field with discriminant $\Delta$. For $\alpha \in K$, the conjugates are denoted by $\alpha'$ and
    $\alpha''$, respectively. Furthermore, for $\alpha \in K$, let $\textup{Tr}(\alpha):=\alpha+\alpha'+\alpha''$ and $\textup{N}(\alpha):=\alpha\cdot\alpha'\cdot\alpha''$.
    Let $\widehat{K}:=K(\sqrt{\Delta}),k\in \mathbb{N}, k\geq2$, and $\{\epsilon_1,\epsilon_2\}$ be a system of fundamental units of $K$. Define $L$ by $L:=\textup{ln}|\epsilon_1/\epsilon_1''|\,\textup{ln}|\epsilon_2'/\epsilon_2''|-
    \textup{ln}|\epsilon_1'/\epsilon_1''|\,\textup{ln}|\epsilon_2/\epsilon_2''|$. Let $W$ be an integral ideal of $K$ with basis $\{w_1,w_2,w_3\}$.
    Let $\rho=\widetilde{w_3}$ for a dual basis $\widetilde{w_1},\widetilde{w_2},\widetilde{w_3}$ of $W$ subject to $$\textup{Tr}(w_i \widetilde{w}_j)=\delta_{ij}(1\leq i, j \leq 3).$$
    For $j=1,2$, set
    \begin{displaymath}
    E_j=\left(\begin{array}{rrr} 1 & 1 & 1\\ \epsilon_j & \epsilon_j' & \epsilon_j''\\ \epsilon_1\epsilon_2 & \epsilon_1'\epsilon_2' & \epsilon_1''\epsilon_2''\end{array}\right)
    \end{displaymath}
    and
    \begin{displaymath}
    B_\rho=\left(\begin{array}{rrr} \rho w_1 & \rho w_2 & \rho w_3\\
     \rho' w_1' & \rho' w_2' & \rho' w_3'\\
     \rho'' w_1'' & \rho'' w_2'' & \rho'' w_3''
      \end{array}\right).
    \end{displaymath}
    For $\tau_1, \tau_2 \in K$, $\nu =1,2$, set
    $$M(k,\nu,\tau_1,\tau_2):=0$$ if $\textup{det} E_\nu=0$, otherwise

    \begin{eqnarray*}
    \lefteqn{M(k,\nu,\tau_1,\tau_2)} \\
     & &: = \textup{sign}(L)(-1)^\nu [\widehat{K}:\mathbb{Q}]^{-1}\frac{(2\pi i)^{3k}}{(3k)!}\textup{N}(\rho)^k\\
     & & \ \ \ \cdot \sum_{m_1=0}^{3k} \sum_{m_2=0}^{3k} \binom{3k}{m_1,m_2}\\
     & & \ \ \ \cdot \{\frac{\textup{det}E_\nu}{|\textup{det}(E_\nu B_\rho)|^3} \mathbf{B}(3,m_1,m_2,3k-(m_1+m_2),(E_\nu B_\rho)^*,\mathbf{0})\\
     & & \ \ \ \cdot \sum_{\kappa_1=0}^{k-1} \sum_{\kappa_2=0}^{k-1} \sum_{\mu_1=0}^{k-1} \sum_{\mu_2=0}^{k-1} \binom{m_1-1}{k-1-(\kappa_1+\kappa_2),k-1-(\mu_1+\mu_2)}\\
     & & \ \ \ \cdot \binom{m_2-1}{\kappa_1,\mu_1} \binom{3k-1-(m_1+m_2)}{\kappa_2,\mu_2}\\
     & & \ \ \ \cdot \textup{Tr}_{\widehat{K}/\mathbb{Q}}(\tau_1^{\kappa_1+\kappa_2} \tau_1'\,^{\mu_1+\mu_2} \tau_1''\,^{3k-2-(m_1+\kappa_1+\kappa_2+\mu_1+\mu_2)}\\
     & & \ \ \ \cdot \tau_2^{\kappa_2} \tau_2'\,^{\mu_2} \tau_2''\,^{3k-1-(m_1+m_2+\kappa_2+\mu_2)})\},
    \end{eqnarray*}
    where  $(E_\nu B_\rho)^*$ denotes the transposed matrix of $(E_\nu B_\rho)$, and

    \begin{eqnarray*}
    \lefteqn{C(k,\nu,\tau_1,\tau_2)} \\
    & & : = \textup{sign}(L)(-1)^{\nu+1}\frac{(2\pi i)^{3k}}{12\cdot(3k-2)(k-1)!^3 }\textup{N}(\rho)^k\\
    & & \ \ \ \cdot \widetilde{B}_{3k-2}(0)|\textup{det}B_\rho|^{-1}\textup{sign}(\textup{det}E_\nu)\\
    & & \ \ \ \cdot\{\textup{sign}((\tau_1\tau_2-\tau_1'\tau_2')(\tau_1-\tau_1'))+\textup{sign}((\tau_1'\tau_2'-\tau_1''\tau_2'')(\tau_1'-\tau_1''))\\
    & & \ \ \ + \textup{sign}((\tau_1''\tau_2''-\tau_1\tau_2)(\tau_1''-\tau_1))+\textup{sign}(\tau_1''(\tau_1-\tau_1')(\tau_2'-\tau_2))\\
    & & \ \ \ + \textup{sign}(\tau_1(\tau_1'-\tau_1'')(\tau_2''-\tau_2'))+\textup{sign}(\tau_1'(\tau_1''-\tau_1)(\tau_2-\tau_2''))\\
    & & \ \ \ + \textup{N}(\tau_2)[\textup{sign}(\tau_1''(\tau_2-\tau_2')(\tau_1\tau_2-\tau_1'\tau_2'))\\
    & & \ \ \ + \textup{sign}(\tau_1(\tau_2'-\tau_2'')(\tau_1'\tau_2'-\tau_1''\tau_2''))+ \textup{sign}(\tau_1'(\tau_2''-\tau_2)(\tau_1''\tau_2''-\tau_1\tau_2))]\}.
    \end{eqnarray*}
    Define
   \begin{eqnarray*}
   \zeta(k,W,\epsilon_1,\epsilon_2):= M(k,1,\epsilon_1,\epsilon_2)+M(k,2,\epsilon_2,\epsilon_1)\\
     + C(k,1,\epsilon_1,\epsilon_2)+C(k,2,\epsilon_2,\epsilon_1).
   \end{eqnarray*}
   Let $\zeta_K(s,K_0)$ be the partial zeta function of an absolute ideal class $K_0$ of $K$ and $W\in K_0^{-1}$. Then we have
   \begin{equation}\label{eqn:1}
   \zeta_K(2k,K_0)=\frac{1}{2}\textup{Norm}(W)^{2k}\zeta(2k,W,\epsilon_1,\epsilon_2).
   \end{equation}
   \end{thm}

 \textit{Remark} 1.
     For $k,l,m \in \mathbb{Z}$,
    \begin{equation*}
    \binom{k}{l,m}:= \begin{cases}\frac{k!}{l! m! (k-(l+m))!} &  \textup{if} \ \ k,l,m,k-(l+m)\in \mathbb{N}\cup \{0\}\\
                                  (-1)^{l+m}\binom{l+m}{l} &  \textup{if} \ \ k=-1 \ \textup{and} \ \ l,m\in \mathbb{N}\cup \{0\}\\
                                   0 &  \textup{otherwise}.\\
                     \end{cases}
    \end{equation*}

    \textit{Remark} 2.
    Let $A=(a_{ij})_{n,n}$ be a regular $(n,n)$-matrix with integral coefficients, $(A_{ij})_{n,n}:=(\textup{det}A)A^{-1}$. Let
    \begin{equation*}
    \widetilde{B}_r(x):=\begin{cases} B_r(x-[x]) & \textup{if} \ \ r=0 \,\, \textup{or} \,\, r\geq2 \,\, \textup{or} \,\, r=1\wedge x\in \!\!\!\!\! / \ \mathbb{Z} \\
                                      0 &  \textup{if} \ \ r=1\wedge x\in \mathbb{Z},
                        \end{cases}
    \end{equation*}
    where $B_r(y)$ is defined as usual by $ze^{yz}(e^z-1)^{-1}=\sum_{r=0}^\infty B_r(y)z^r/r!$. Then, for $\mathbf{r}=(r_1, \ldots, r_n)\in (\mathbb{N}\cup \{0\})^n$,

    \begin{equation*}
    \mathbf{B}(n,\mathbf{r},A,\mathbf{0})=\sum_{\kappa_1=0}^{|\textup{det}A|-1}\cdots\sum_{\kappa_n=0}^{|\textup{det}A|-1}\prod_{i=1}^n
    \widetilde{B}_{r_i}(\frac{1}{\textup{det}A}\sum_{j=1}^n A_{ij}\kappa_j).
    \end{equation*}

Next, we introduce Siegel's formula for
the values of the Dedekind zeta function of a totally
real algebraic number field at negative odd integers.

For an ideal $I$ of the ring of integers $\mathcal{O}_K$, we define the sum
of ideal divisors function $\sigma_r(I)$ by

\begin{equation}
  \sigma_r(I)=\sum_{J \mid I} N_{K/\mathbb{Q}}(J)^r,
\end{equation}
where $J$ runs over all ideals of $\mathcal{O}_K$ which divide $I$.
Note that, if $K=\mathbb{Q}$ and $I=(n)$, our definition coincides with
the usual sum of divisors function

\begin{equation}
  \sigma_r(n)=\sum_{d \mid n \atop d>0} d^r.
\end{equation}
Now let $K$ be a totally real algebraic number field. For
$l,b=1,2, \ldots ,$ we define

\begin{equation}\label{eqn:6}
  S_l^{K}(2b)= \sum_{
  \begin{smallmatrix}
    \nu \in \mathcal{D}_K^{-1} \\
              \nu \gg\,  0 \\
    \scriptsize{\textup{Tr}_{K/\mathbb{Q}}(\nu)=\, l} \\
  \end{smallmatrix}} \sigma_{2b-1}((\nu)\mathcal{D}_K),
\end{equation}
where $\mathcal{D}_K$ is the different of $K$.
At this moment, we remark that this is a finite sum.
Siegel\cite{Si} proved:

\begin{thm}\label{theorem:1.2}
Let $b$ be a natural number, $K$ a totally real algebraic
number field of degree $n$, and $h=2bn$. Then

  \begin{equation}
    \zeta_K(1-2b)=2^n \sum_{l=1}^r b_l(h)S_l^K(2b).
  \end{equation}
The numbers $r\geq 1$ and $b_1(h), \ldots , b_r(h) \in \mathbb{Q}$
depend only on $h$. In particular,
  \begin{equation}
    r = \textup{dim}_\mathbb{C}\mathcal{M}_h,
  \end{equation}
where $\mathcal{M}_h$ denotes the space of modular forms of weight
$h$. Thus by a well-known formula,

  \begin{equation*}
   r=\begin{cases} [\frac{h}{12}] & \mbox{ if } \ \
                 h \equiv  2\ (\mbox{\rm mod}\, 12)\\
                [\frac{h}{12}]+1 & \mbox{ if } \ \ h \equiv \!\!\!\!\! / \ 2\ (\mbox{\rm mod}\, 12).\\
      \end{cases}
  \end{equation*}
\end{thm}

    Now, we will introduce our target fields. Let $m (\geq 4)$ be a rational integer and $K_m$(or simply $K$)$=\mathbb{Q}(\alpha)$
    be the non-normal totally real cubic number field
    (whose arithmetic was studied in \cite{Lou1}) associated with the irreducible cubic polynomial
    \begin{equation}\label{eqn:2}
                f_m(x)=x^3-mx^2-(m+1)x-1\in \mathbb{Z}[x]
    \end{equation}
    of positive discriminant
    \begin{equation*}
    D_m=(m^2+m-3)^2-32>0
    \end{equation*}
    and with three distinct real roots $\alpha_3 < \alpha_2 < \alpha_1 = \alpha$.
    We borrow known results for arithmetic of $K_m$.
      \begin{thm}\label{theorem:1.3}

          \begin{enumerate}
          \item[{\rm(1)}]
                The set $\{1, \alpha, {\alpha}^2\}$ forms an integral basis of the
                ring $\mathcal{O}_K$ of algebraic integers of $K$ if and only if one of the
                following conditions holds true:
                \begin{enumerate}
                    \item[{\rm(i)}] $m \equiv \!\!\!\!\! / \,\, 3\,(\mbox{\rm mod}\, 7)$ and $D_m$ is
                    square-free,

                    \item[{\rm(ii)}] $m \equiv 3\,(\mbox{\rm mod}\,7)$, $m \equiv \!\!\!\!\! / \,\,
                    24\,(\mbox{\rm mod}\, 7^2)$ and $\frac{D_m}{7^2}$ is
                    square-free.
                 \end{enumerate}
          \item[{\rm(2)}]
          The full group of algebraic units of $K_m$ is ${<}{-1},\alpha,{\alpha+1}{>}$.
                 \end{enumerate}
            \end{thm}

        \begin{proof}
            See \cite{Lou1}.
        \end{proof}

 \section{class number one criterion for $K_m$}
       In this section, to have the value of $\zeta_K(-1,P)$, we apply Theorem \ref{theorem:1.1} to $K_m$. On the other hand,
       we evaluate the upper bound of $\zeta_K(-1)$ by using Theorem \ref{theorem:1.2}. Finally, combining these results,
       we give a class number one criterion for $K_m$.

       We take $W=\mathcal{O}_K=(\alpha)$. Since the ideal class containing $\mathcal{O}_K$ is the principal ideal class $P$,
       by $(\ref{eqn:1})$, we have

       \begin{equation*}
       \zeta_K(2,P)=\frac{1}{2}\, \zeta(2,\mathcal{O}_K,\alpha,{\alpha+1}).
       \end{equation*}
      By definition,
      \begin{eqnarray*}
   \zeta(2,\mathcal{O}_K,\alpha,\alpha+1)= M(2,1,\alpha,\alpha+1)+M(2,2,\alpha+1,\alpha)\\
     + C(2,1,\alpha,\alpha+1)+C(2,2,\alpha+1,\alpha).
   \end{eqnarray*}
   Let $\{\widetilde{w_1},\widetilde{w_2},\widetilde{w_3}\}$ be a dual basis of $\mathcal{O}_K$.
   Then, by a simple computation, we get
    \begin{eqnarray*}
   \rho=\widetilde{w_3}=\frac{-1}{D_{m}}\{(m^3+5m^2+5m+4)+(2m^3+7m^2+7m+9)\alpha\\-2(m^2+3m+3)\alpha^2\}
   \end{eqnarray*}
   This makes it possible to determine matrices $E_1, E_2$ and $B_\rho$.
   Now, we note that 3-fold Dedekind sum $\mathbf{B}(3,m_1,m_2,6-(m_1+m_2),(E_\nu B_\rho)^*,\mathbf{0})$ vanishes
   when $m_1$ or $m_2$ is odd. Next, we need the computation for trace. This computation is very long but elementary.
   Combining these data, we have
   \begin{eqnarray*}
      \lefteqn{M(2,1,\alpha,\alpha+1)=-(4m^9+54m^8+304m^7+979m^6}\\\nonumber
      &&+2119m^5+3234m^4+3327m^3+2067m^2+72m-714)\pi^6/2835D_m^{3/2}
   \end{eqnarray*}
    \begin{eqnarray*}
      \lefteqn{M(2,2,\alpha+1,\alpha)=(4m^9+54m^8+304m^7+985m^6}\\\nonumber
      &&+2137m^5+3204m^4+3237m^3+2091m^2+144m-714)\pi^6/2835D_m^{3/2}
   \end{eqnarray*}
   On the other hand, the calculation of $C(2,1,\alpha,\alpha+1)$(resp.~$C(2,2,\alpha+1,\alpha)$) is
   simpler than one of $M(2,1,\alpha,\alpha+1)$(resp.~$M(2,2,\alpha+1,\alpha)$).
   In fact,
   \begin{equation*}
   C(2,1,\alpha,\alpha+1)=\frac{2\pi^6}{45D_m^{3/2}}, \  C(2,2,\alpha+1,\alpha)=-\frac{2\pi^6}{45D_m^{3/2}}.
   \end{equation*}
   Then, by collecting these results, we have the following theorem.
        \begin{thm}
       Let $m(\geq4)$ be an integer which satisfies the conditions of Theorem $\ref{theorem:1.3}$ and $K_m$ the non-normal totally real
       cubic field defined by $(\ref{eqn:2})$. Let $P$ be the principal ideal class of $K_m$. Then we have
       \begin{equation*}
       \zeta_K(2,P)=\frac{m(m^5+3m^4-5m^3-15m^2+4m+12)\pi^6}{945(D_m)^{3/2}}.
       \end{equation*}
       Moreover, by a functional equation,
        \begin{equation*}
        \zeta_K(-1,P)=-\frac{m(m^5+3m^4-5m^3-15m^2+4m+12)}{7560}.
        \end{equation*}
       \end{thm}
       Next, by Theorem~\ref{theorem:1.2}, noting that $b_1(8)=-1/504$(cf.\ \cite{Za}), we have
       \begin{equation*}
       \zeta_K(-1)=-\frac{8}{504} S_1^K(2)=-\frac{8}{504}\sum_{\nu \in S_1} \sigma_1((\nu) \mathcal{D}_K),
       \end{equation*}
       where
       \begin{equation*}
       S_1:=\{\nu \in K|\, \nu \in \mathcal{D}_K^{-1}, \nu \gg 0, \textup{Tr}_{K/\mathbb{Q}}(\nu)=1\}.
       \end{equation*}
       Let $T$ be the set of integral points in $(s,t)$-plane corresponding to $S_1$ by one-to-one correspondence in [4, Proposition 2.1].
       This set has been completely determined in [4, Theorem 2.3] as follows:

        \begin{eqnarray*}
                       T \;=&\{& (1,1),\hspace{0.5cm}(1,2),\hspace{1cm}
                        \ldots\hspace{1cm},\hspace{0.5cm} (1,m-1),\\
                        && (2,2),\hspace{0.5cm}(2,3),\hspace{1cm}\ldots
                        \hspace{1cm},\hspace{0.5cm}(2,m),\\
                        && (3,3),\hspace{0.5cm}(3,4),\hspace{1cm}\ldots
                        \hspace{1cm},\hspace{0.5cm}(3,m),\\
                        && \ldots\ldots\hspace{2cm}\ldots\ldots
                        \hspace{1.5cm}\ldots\ldots,\\
                        && (m-2,m-2),(m-2,m-1),(m-2,m),\\
                        && (m-1,m) \: \}.\hspace{2cm}
                    \end{eqnarray*}
        Furthermore, by (26) of \cite{CKL}
       \begin{equation*}
       N((\nu) \mathcal{D}_K)= |f_m(s,t)|,
       \end{equation*}
       where
        \begin{eqnarray*}
                   \lefteqn{ f_m(s,t)=(-s^2+(t+1)s)m^2+((t-2)s^2-(t^2-t)s-(t^2+t))m}\\\nonumber
          & & \ \ \ \ \ \ \ \ \ \ \ \ +(s^3-2s^2-(t^2-3t-1)s+t^3-t-1).
                \end{eqnarray*}
       One can easily check that $f_m(s,t) > 1$ for all $(s,t)\in T$.
       Therefore, we have the following inequalities
         \begin{eqnarray}\label{eqn:3}
                  \ \ \ \ \zeta_K(-1) & \leq & -\frac{8}{504}\sum_{\nu \in S_1}(1+N((\nu)\mathcal{D}_K))\\\nonumber
                   & = & -\frac{8}{504}\{\sharp S_1+\sum_{\nu \in S_1}N((\nu)\mathcal{D}_K)\}\\\nonumber
                   & = & -\frac{8}{504}\{\frac{1}{2}(m^2+m-6)+\sum_{(s,t) \in T}f_m(s,t)\}\\\nonumber
                   & = & -\frac{8}{504}\{\frac{1}{2}(m^2+m-6)+\sum_{t=1}^{m-1}f_m(1,t)\\\nonumber
                   & & \ \ \ \ \ \ \ \ \ +\sum_{s=2}^{m-2} \sum_{t=s}^m f_m(s,t)+f_m(m-1,m)\}\\\nonumber
                   & = &-\frac{m(m^5+3m^4-5m^3-15m^2+4m+12)}{7560} \\\nonumber
                   & = & \zeta_K(-1,P),
         \end{eqnarray}
         and equality holds in $(\ref{eqn:3})$ if and only if $(\nu) \mathcal{D}_K$ is a prime ideal for all $\nu \in S_1$.
          Combining  this computation, we give a class number one criterion for $K_m$.

          \begin{thm}
          Let $m(\geq4)$ be an integer which satisfies the conditions of Theorem $\ref{theorem:1.3}$ and $K_m$ the non-normal totally real
       cubic field defined by $(\ref{eqn:2})$. Then we have
       \begin{equation*}
       h_K=1 \ \textup{if and only if} \ (\nu) \mathcal{D}_K \ \textup{is a prime ideal for all} \ \nu \in S_1.
       \end{equation*}
        \end{thm}
        On the other hand, Louboutin\cite{Lou1} showed:
         $$m=4,5,6,8\ \textup{gives all the values of} \ m \ \textup{such that}\ h_K=1.$$
       Therefore, we can conclude that
        \begin{equation*}
       m=4,5,6,8 \ \textup{if and only if} \ (\nu) \mathcal{D}_K \ \textup{is a prime ideal for all} \ \nu \in S_1.
       \end{equation*}

     \textit{Remark} 3.
      Unlike in the simplest cubic fields which is a Galois extension of $\mathbb{Q}$, $N((\nu) \mathcal{D}_K)= |f_m(s,t)|$ is necessarily not prime
      where $(\nu) \mathcal{D}_K$ is a prime ideal for each $\nu \in S_1$.
      For example, $f_m(2,3)=({2m-5})^2$ for the point $(2,3)$ in $T$, but if $m=4,5,6,8$, then each integral ideal $(\nu) \mathcal{D}_K$ corresponding to the
      point $(2,3)$ is prime. Furthermore, $f_m(s,t)$ is a prime for all points only except $(2,3)$ in $T$ when $m=4,5,6,8$.

\textbf{Acknowledgements.}
The author would like to thank the referee for careful reading of the manuscript and helpful comments.

\end{document}